\documentclass[final,3p,12pt]{elsarticle}

\usepackage{graphicx}
\usepackage{amssymb,stmaryrd}
\usepackage{amsthm}
\usepackage{amsmath}
\usepackage{color}
\usepackage{comment}
\usepackage{calligra}
\DeclareMathAlphabet{\mathcalligra}{T1}{calligra}{m}{n}
\DeclareFontShape{T1}{calligra}{m}{n}{<->s*[2.2]callig15}{}

\journal{Discrete Mathematics}

\newcommand{\ml}{l\kern-0.55mm\char39\kern-0.3mm}

\newtheorem{theorem}{Theorem}
\newtheorem{corollary}{Corollary}
\newtheorem{conjecture}{Conjecture}
\newtheorem{problem}{Problem}

\begin{document}

\begin{frontmatter}


\title{Betweenness centers of graphs}



\author[TM]{Tomáš Madaras\corref{TMc}}
\ead{tomas.madaras@upjs.sk}
\cortext[TMc]{Corresponding author}
\author[MP]{Matúš Paralič}
\ead{matus.paralic@student.upjs.sk}

\address[TM,MP]{P.J. Šafárik University in Košice, Slovakia}

\begin{abstract}
The betweenness centrality of a vertex $v$ in a graph $G = (V,E)$ is the sum of the relative numbers of shortest paths of $G$ that pass through $v$. The vertices of $G$ which have the maximum (resp. minimum) betweenness induce the betweenness center (resp. betweenness periphery) of $G$. We study betweenness of graphs and their localization in graph blocks, presenting sufficient conditions for graphs (in terms of diameter or block sizes) to have those centers contained in a single block. Further, we show that each graph occurs as the subgraph induced by the betweenness center of some graph (as well as the subgraph induced by the betweenness periphery). For trees, we show, by an alternative proof, that their betweenness center is always contained in a path; in addition, we enumerate trees of order at most 20 according to the order of their betweenness centers.
\end{abstract}


\end{frontmatter}


\section{Introduction}
\label{S:1}

Throughout this paper, we consider (unless stated otherwise) connected graphs without loops or multiple edges; the terminology is taken from \cite{West}.

In research on social or complex networks, a lot of attention is paid to the identification of network actors that play an important (or central) role within a network; typical examples include popular and influential persons in communities, super-spreaders of infectious diseases, or critical facilities of a country's power network. The quantitative measure of actor importance is expressed in terms of so-called centrality indices. Formally, a {\em vertex centrality} is a function $c:\ V(G) \to \mathbb{R}_0^+$ which is invariant under graph isomorphism and reflects the subjective perception of central vertices in a graph (in the sense that the more the vertices are perceived as central -- like the center of a star graph compared to its pendant vertices -- the higher values of $c$ they have). A function defined on $V(G)$ (not necessarily positive) which is just invariant under graph isomorphism is called a {\em structural index}. An example of a structural index which is not a centrality is the vertex eccentricity ${\rm ecc}(v) = \max\limits_{x \in V(G)} d(v,x)$ (however, one can take its reciprocal to conform with the concept); among classical centralities, one finds vertex degree, the closeness centrality ${\rm cl}_G(v) = \dfrac{1}{\sum\limits_{x \in V(G)} d(v,x)}$, and the (unnormalized) {\em betweenness centrality} ${\rm b}_G(v) = \displaystyle\sum\limits_{\{x,y\}\in \binom{V(G)\setminus{v}}{2}} \dfrac{\sigma_{x,y}(v)}{\sigma_{x,y}}$ (the subscript is omitted if $G$ is known from the context), where $\sigma_{x,y}$ is the total number of shortest $x-y$-paths in $G$ and $\sigma_{x,y}(v)$ is the number of those shortest $x-y$-paths having $v$ as an internal vertex. The betweenness centrality is frequently used in the analysis of social and complex networks (see, for example, \cite{SaxenaIyengar}, Section 5.5, or \cite{GirvanNewman,Pacheco,ZaoliEtAl,ZhaoEtAl}), but, within the last two decades, there has also been an increase in studies of its graph-theoretical properties, as in \cite{GagoHurajovaMadaras,GagoHurajovaMadaras2,GagoHurajovaMadaras3,HartmanPokornaValtr,GhanbariEtAl}.

Every centrality index $c$ induces, in a graph $G = (V,E)$, two important induced subgraphs: the $c$-center $\mathcal{C}_c(G)$ induced by the set $C_c(G) = \{v \in V(G):\ c(v) = \max\limits_{x \in V(G)} c(x)\}$, and the $c$-periphery induced by the set $P_c(G) = \{v \in V(G):\ c(v) = \min\limits_{x \in V(G)} c(x)\}$; if $c$-center of $G$ coincides with $G$, then $G$ is called $c$-uniform.  When studying these subgraphs, one may (typically) discuss the following statements:
\begin{itemize}
    \item[(S1)] The $c$-center of $G$ lies in a single block.
    \item[(S2)] Every $c$-uniform graph is 2-connected.
    \item[(S3)] Given a graph $H$, there exists a graph $G_1$ (resp. $G_2$) whose $c$-center (resp. $c$-periphery) is isomorphic to $H$. 
\end{itemize}

Obviously, (S1) implies (S2), and (S1) holds for several important centrality indices, namely, for the reciprocal of eccentricity (\cite{HararyNorman}) as well as for closeness centrality (\cite{Truszczynski}), and for certain less known indices, like the reciprocal of detour eccentricity (\cite{ChartrandEtAl}). Both (S1) and (S2) fail for vertex degree (although a slightly weaker analogue of (S2) holds for even-degree regular graphs: they are bridgeless, that is, edge-2-connected). For betweenness, (S2) was proved in \cite{GagoHurajovaMadaras}, but, in general, (S1) does not hold: in Section 2, we present an infinite family of counterexamples. Nevertheless, based on the specific structure of these counterexamples, we prove, in Section 2, several sufficient conditions (in terms of diameter, number of edges or block orders) which guarantee the validity of (S1); some of these results are sharp.

Our interest in (S3) is motivated by its validity for the classical eccentricity-based center (\cite{Hedetniemi}), for the median center (\cite{Slater}), for the detour center (\cite{ChartrandEtAl}) and for the periphery (\cite{BielakSyslo}; here (S3) does not hold in general, but the complete characterization is known). In Section 3, we provide graph constructions showing the full validity of (S3) both for the betweenness center and periphery. In Section 4, we give an alternative proof of the fact (previously known, but not in the context of betweenness) that, in trees, the betweenness center lies within a path, and present the census of trees up to 20 vertices with respect to the orders of their betweenness centers.

\section{Localization of betweenness centers in blocks}

We first present two examples of small graphs, and an infinite family of graphs such that their betweenness center is not contained in a single block (thus, is {\em delocalized}).

\bigskip
\noindent
\textbf{Example 1:} Let $T$ be the complete binary tree of height 2 (that is, the 7-vertex tree with two 3-valent vertices $u,w$ , central 2-valent vertex $v$ and four pendant vertices). The direct calculation yields ${\rm b}(v) = 3\cdot 3 = 9$, ${\rm b}(u) = {\rm b}(w) = 2\cdot 4 + 1 = 9$; hence $u,v,w$ induce the betweenness center of $T$ which lies in two blocks of $T$. An exhaustive computer search over all non-isomorphic connected graphs of order at most six shows that each of them has its betweenness center localized within a single block. 

\medskip
\noindent \textbf{Example 2:} Let $F_1,F_2$ be two copies of a fan $F_{1,9} \cong P_9 \vee K_1$ (that is, the join of a vertex $y$ and a 9-vertex path). Let $x_i \in V(F_i)$ be a 2-valent vertex in such a copy, $y_i$ be its central vertex of degree 9, and let $F$ be a graph obtained from $F_1,F_2$ by identifying $x_1,x_2$ into a single common 4-valent vertex $x$.

\bigskip
\begin{figure}[h!]
    \centering
    \includegraphics[width=0.7\linewidth]{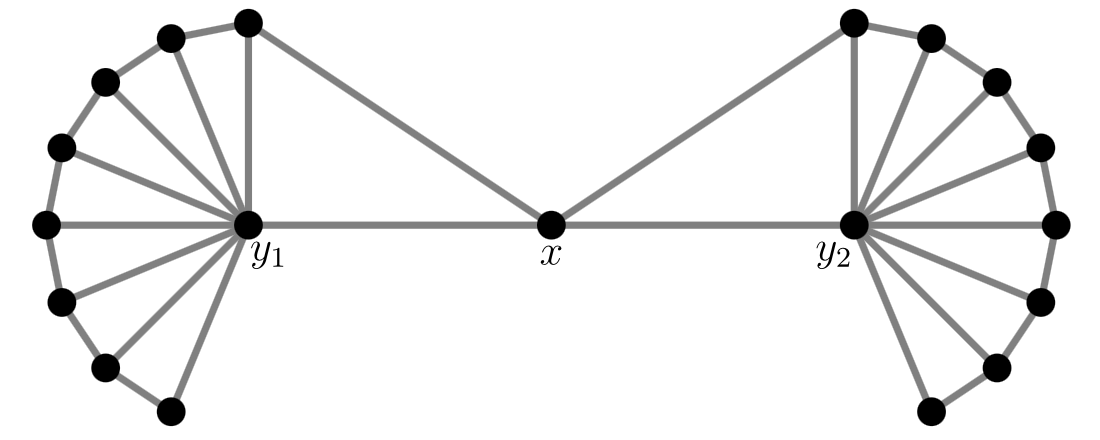}
    \caption{The graph $F$}
    \label{fig:enter-label}
\end{figure}


\medskip
Because both neighbors of $x$ in the subgraphs of $F$ which are isomorphic to $F_1,F_2$ are adjacent, a positive contribution to the value of ${\rm b}(x)$ comes only from those pairs $\{a,b\}$ of vertices of $F$ such that $a \in V(F_1)$, $b\in V(F_2)$ and $a,b \neq x$; hence ${\rm b}(x) = 9\cdot 9 = 81$. On the other hand, the value ${\rm b}(y_1)$ equals the sum of ${\rm b}_{F_1}(y_1)$ and the contributions from the pairs $(a,b)$ such that $b \in V(F_2)\setminus \{x\}$. By  \cite{KlisaraHurajovaMadarasSkrekovski} (Proposition 3.2 on the maximum of betweenness centrality in 2-connected graphs),  ${\rm b}_{F_1}(y_1) = \frac{(10-3)^2}{2}$; the remaining contributions to ${\rm b}(y_1)$ give $(10-4)\cdot(10-1) + \frac{1}{2}\cdot 1 \cdot 9$ in total. Thus ${\rm b}(y_1) = 83 = {\rm b}(y_2)$, and so the betweenness center of $F$ is delocalized.  

\medskip
\noindent \textbf{Example 3:}
Let $k \geq 2$ be an integer, let $P_{u,v}$ be a path of length $k$ with endvertices $u$ and $v$. Let $T_k$ be a tree obtained from $P_{u,v}$ by attaching \(k+1\) new pendant vertices to both $u$ and $v$. Let $w$ be a vertex on $P_{uv}$ such that $d(u,w) = \ell, d(v,w) = k-\ell$ (by symmetry, one can assume that $0 \leq \ell \leq \left \lfloor \frac{k}{2} \right \rfloor$). If $w \neq u,v$, then  $T_k - w$ consists of two trees on $k+1 + \ell$ and $k+1 + k -\ell$ vertices; by \cite{GagoHurajovaMadaras} (Theorem 3), we have ${\rm b}(w) = (k+1+\ell)(k+1+k-\ell)$. Thus, with $k$ being fixed, ${\rm b}(w)$ is a quadratic function of $\ell$ which is maximized for $\ell = \left \lfloor \frac{k}{2} \right \rfloor$; the maximum then equals $(k+1 + \left \lfloor \frac{k}{2} \right \rfloor)(k + 1 + \left \lceil \frac{k}{2} \right \rceil)$. Now, let $w = u$. Then ${\rm b}(w) = {{k+1}\choose 2} + (k+1)\cdot 1 \cdot (k+1+k) = \frac{5k^2+7k+2}{2}$. If $k$ is even, we have $(k+1 + \left \lfloor \frac{k}{2} \right \rfloor)(k + 1 + \left \lceil \frac{k}{2} \right \rceil) = \left ( \frac{3}{2}k +1\right )^2 < \frac{5k^2+7k+2}{2}$ for $k \geq 2$; otherwise $(k+1 + \left \lfloor \frac{k}{2} \right \rfloor)(k + 1 + \left \lceil \frac{k}{2} \right \rceil) = \frac{3(3k^2+4k+1)}{4}< \frac{5k^2+7k+2}{2}$. Thus, in $T_k$, the maximum betweenness is attained for $u,v$ (note that their distance is $k$).

\medskip
Observe that the graphs from Examples 1 and 2 have diameter 4; this value is sharp, due to the following

\begin{theorem} 
\label{betweenness_diam_3}
Let $G$ be a graph of diameter at most 3. Then its betweenness center lies in a single block.
\end{theorem}

\noindent \textbf{Proof:} The statement is obvious for 2-connected graphs. First, assume that $\operatorname{diam}(G)=2$ and $G$ is not 2-connected. Then $G$ has a unique cut-vertex $a$ which is adjacent to every other vertex of $G$. For a vertex $x \neq a$ in a block $B$, the contribution of any pair of nonadjacent vertices $u,v \in V(B)$ to ${\rm b}(x)$ is not greater than its contribution to ${\rm b}(a)$, since whenever $u-x-v$ is a shortest $u$-$v$ path, so is $u-a-v$. On the other hand, pairs of vertices lying in distinct components of $G-a$ contribute to ${\rm b}(a)$ but not to ${\rm b}(x)$. Therefore ${\rm b}(a)>{\rm b}(x)$.

Now, assume that $\operatorname{diam}(G)=3$ and $G$ is not 2-connected. Then all cut-vertices of $G$ induce a clique (otherwise $G$ contains two nonadjacent cut-vertices $x,y$ and a path of length at least $1+d(x,y)+1\geq4$, a contradiction), and hence they belong to a single block $B$ of $G$.

If an endblock $B'$ with cut-vertex $a$ contains a vertex $x$ which is not adjacent to $a$, then every vertex outside $B'$ is adjacent to $a$; otherwise its distance from $x$ would be at least $4$. By the same comparison of contributions as above, ${\rm b}(y)<{\rm b}(a)$ for every $y\notin B'$, and hence the betweenness center is contained in $B'$.

Otherwise, in every endblock all vertices are adjacent to its cut-vertex. For a vertex $x\notin B$, let $a\in B$ be the cut-vertex of the endblock containing $x$. Again, the same argument as in the diameter-2 case gives ${\rm b}(x)<{\rm b}(a)$. Thus no vertex outside $B$ has maximum betweenness, and the betweenness center is contained in $B$.
\hfill $\square$
\bigskip

\begin{corollary}
Let $G$ be an $n$-vertex graph with at least $4+\frac{n(n-5)}{2}$ edges. Then its betweenness center is contained in a single block.
\end{corollary}

\noindent \textbf{Proof:}
By \cite{Ore}, the maximum number of edges of an $n$-vertex graph of diameter $d$ is \(\displaystyle d+\frac{(n-d-1)(n-d+4)}{2}\) (note that $n\geq d+1$). For $d=4$, this yields that every $n$-vertex graph with more than \(\displaystyle 4+\frac{n(n-5)}{2}\) edges has diameter at most 3, and so its betweenness center lies in a single block. From \cite{Ore}, one also obtains a characterization of maximal $n$-vertex graphs of diameter 4 (with $4+\frac{n(n-5)}{2}$ edges): each such graph is obtained either from $K_{n-2}$ by deleting an edge $ab$ and adding a new path $avu$, or from the join of two complete graphs $K_\ell$ and $K_{n-3-\ell}$ in such a way that a new vertex $z$ is joined to all vertices of $K_{n-3-\ell}$, a new vertex $u$ is joined to all vertices of $K_\ell$, and a third vertex $v$ is joined to $u$. The latter graph has only two blocks, one of them being $K_2$; hence its betweenness center lies in a single block. For the former extremal graph, we have ${\rm b}(u)=0$ (since $u$ is a pendant vertex), ${\rm b}(v)=n-2$, and ${\rm b}(a)=2n-6$. Since $n\geq5$, we have $2n-6>n-2$; thus ${\rm b}(a)>{\rm b}(v)$, which yields that the betweenness center again lies in a single block.
\hfill $\square$

\medskip
To discuss the quality of the lower bound on the number of edges in the above result, we describe, for every \(n\geq9\), a construction of an $n$-vertex graph with \(\frac{n^2}{4} - O(n)\) edges whose betweenness center is delocalized.  Consider two copies $K,K'$ of the complete graph  $K_k$, take two vertices $u \in K,u' \in K'$ and join them to a new vertex $x$; in addition, join $u,u'$ with new pendant vertices $v,v'$, respectively. The obtained graph has $n = 2k+3$ vertices, and for any vertex $y \neq u,x,u'$, ${\rm b}(y) = 0$ holds (because the neighbors of $y$ induce a complete subgraph). By symmetry, we also have ${\rm b}(u) = {\rm b}(u')$. Now \({\rm b}(x) = (k+1)^2 = k^2+2k+1\) and \({\rm b}(u) = k \cdot (k+2) + (k-1) \cdot 1 = k^2 + 3k - 1\), while  ${\rm b}(u) \geq {\rm b}(x)$ if and only if \( k^2 + 3k - 1 \geq k^2+2k+1\), which holds for \( k \geq 2\). Hence, the obtained graph has delocalized betweenness center, $n = 2k+3$ vertices and \(2 {k \choose 2} + 4 =  \frac{n-3}{2} \cdot \left(\frac{n-3}{2} - 1\right) + 4 = \frac{n^2}{4}  -2n + \frac{31}{4}\) edges.
A similar construction yielding a graph of even order joins the vertices $u,u'$ by a new path  $uxx'u'$ of length 3; in the graph so obtained, ${\rm b}(y)=0$ for every vertex $y \neq u,x,x',u'$, hence, by symmetry, it is enough to compare the values ${\rm b}(u),{\rm b}(x)$. We have \({\rm b}(x) = (k+1) \cdot (k+2) = k^2+3k+2\) and \({\rm b}(u) = k \cdot (k+3) + (k-1) \cdot 1 = k^2 + 4k - 1\), while  ${\rm b}(u) \geq {\rm b}(x)$ if and only if \( k^2 + 4k - 1 \geq k^2+3k+2\), that is, for \( k \geq 3\). The obtained graph has $n = 2k + 4$ vertices and \(2{k \choose 2} + 5 = \frac{n-4}{2} \cdot \left (\frac{n-4}{2} - 1\right ) + 5 = \frac{n^2}{4} - \frac{5n}{2} + 11\) edges.

\medskip
Example 2 shows that even a graph with two blocks may have a delocalized betweenness center. On the other hand, we prove a positive result in the case that blocks are small:

\begin{theorem} Let $G$ be a graph with a unique cut-vertex and $k$ blocks $B_1,\dots, B_k$, where each of $B_1,\dots, B_{k-1}$ has at most 9 vertices while $B_k$ has at least 9 vertices. Then the betweenness center of $G$ lies in a single block.
\end{theorem}

\noindent \textbf{Proof:} Let $a$ be the unique cut-vertex of $G$ (observe that ${\rm b}(a)>0$) and let $i\in\{1,\dots,k-1\}$ be fixed. Put \(b=|V(B_i)\setminus\{a\}|\) and \(p=|\bigcup\limits_{{j=1}\atop {j\neq i}}^k V(B_j)\setminus\{a\}|\) (note that $b\leq8\leq p$). Assume that there exists $x\in V(B_i)\setminus\{a\}$ such that ${\rm b}(x)=\max\limits_{v\in V(G)}{\rm b}(v)$; we may also assume that $B_i$ is not complete (since then ${\rm b}(x)=0$).

It is easy to see that ${\rm b}(a)\geq bp$, because every path between a vertex from $B_i$ and a vertex from another block passes through $a$. Now we estimate ${\rm b}(x)$. In \cite{KlisaraHurajovaMadarasSkrekovski}, it is proved that the maximum betweenness in an $n$-vertex 2-connected graph is $\frac{(n-3)^2}{2}$. This implies that the contribution to ${\rm b}(x)$ from pairs of vertices which are both from $B_i$ is at most $\frac{(b-2)^2}{2}$. Furthermore, no pair of vertices from the other blocks contributes to ${\rm b}(x)$.

Finally, consider pairs $\{u,v\}$ where $u\in V(B_i)$ and $v$ belongs to another block. If every vertex of $B_i\setminus\{a,x\}$ is adjacent to $a$, then none of these pairs contributes to ${\rm b}(x)$, and ${\rm b}(x)<bp\leq{\rm b}(a)$. Otherwise, let $u\in V(B_i)\setminus\{a,x\}$ be a vertex not adjacent to $a$. Since $B_i$ is 2-connected, $B_i-x$ is connected. Consider a shortest $u$-$a$ path in $B_i-x$, and let $w$ be the neighbor of $a$ on this path and $z$ its preceding vertex. Then $w\neq x$ and $z$ is not adjacent to $a$. Hence neither $w$ nor $x$ can occur as the endpoint in a pair contributing to ${\rm b}(x)$, while for the vertex $z$ the corresponding contribution is at most $\frac12$, since $zwa$ is a shortest $z$-$a$ path avoiding $x$. Therefore, the total contribution of such pairs to ${\rm b}(x)$ is at most $(b-3)p+\frac12p$. From this, we obtain
\[
{\rm b}(x)\leq \frac{(b-2)^2}{2}+(b-3)p+\frac12p=\frac{(b-2)^2}{2}+bp-\frac52p.
\]

The fact that $b\leq8$ gives $(b-2)^2\leq36$; moreover, $p\geq8$, so $\frac52p\geq20$. Thus
\[
\frac{(b-2)^2}{2}\leq18<20\leq\frac52p,
\]
which implies ${\rm b}(x)<bp\leq{\rm b}(a)$, a contradiction to the maximality of the betweenness of $x$. So all vertices with the maximum betweenness are contained in the block $B_k$.
\hfill $\square$

\medskip
Example 2 also shows that the bound 9 on the size of blocks cannot be increased.

\section{Realizations of graphs as betweenness centers and  peripheries}

Here, we show the core results on the nature of betweenness centers and peripheries with respect to their appearance in general graphs:

\begin{theorem}
Let $G$ be a graph. Then there exists a graph $F$ such that the subgraph induced by the betweenness center of $F$ is isomorphic to $G$.
\end{theorem}

\noindent \textbf{Proof:} The statement clearly holds for complete graphs (here, put $F$ = $G$). We start with the well-known fact (see \cite{Sabidussi}) that every (non-complete) graph $G$ is an induced (proper) subgraph of a Cayley graph $H$. Let $A \subseteq V(H)$ be the vertex set of a fixed induced copy of $G$ in $H$. Let $F_0$ be the graph obtained from $H$ in such a way that, for every vertex $v \in V(H)$ we add two new vertices $p_v,q_v$ and three new edges $vp_v,vq_v, p_vq_v$. Finally, let $F$ be the graph obtained from $F_0$ by deleting edges $p_vq_v$ for every $v \in A$; note that $A$ induces, in $F$, a subgraph isomorphic to $G$.

Note that for every $v \in V(H) \subset V(F)$, ${\rm b}_{F}(p_v) = {\rm b}_{F}(q_v) = 0$ as the neighborhood of these vertices induces $K_1$ or $K_2$ in $F$. Further, the vertex transitivity of $H$ implies that, in $F_0$, all vertices from $V(H)$ belong to the same orbit of the automorphism group of $F_0$; consequently, in $F_0$, they have the same value $b$ of betweenness centrality. 

Now, we discuss the change of betweenness in $F$ compared to the values of betweenness in $F_0$. If $a \in A \subset V(F)$, then the pair $p_a,q_a$ contributes 1 to the value of ${\rm b}_F(a)$ (in $F$, there is a unique path $p_a a q_a$ between $p_a$ and $q_a$). The contribution of every other pair $x,y$ to ${\rm b}_F(a)$ is the same as its contribution to ${\rm b}_{F_0}(a)$ since the sets of the shortest $x-y$-paths in $F_0$ and $F$ (as well as the sets of those of them that pass through $a$) are the same (note that, for $a \in A$, the vertices $p_a,q_a$ of $F_0$ never belong to a shortest $x-y$ path of $F$ unless $\{x,y\} = \{p_a,q_a\}$). An analogous argument applies also to a vertex $u \in V(H) \setminus A$ of $F$ (but here, the pair $p_u,q_u$ does not contribute to ${\rm b}_F(u)$). Therefore, we conclude that ${\rm b}_{F}(a) = b+1$ for $a \in A$, ${\rm b}_{F}(u) = b$ for $u \in V(H) \setminus A$, and ${\rm b}_F(v) = 0$ for every other vertex $v$ of $F$. Thus, the set $A$ is the betweenness center of $F$.
\hfill\qed

\begin{theorem}
Let $G$ be a graph. Then there exists a graph $F$ such that the subgraph induced by the betweenness periphery of $F$ is isomorphic to $G$.
\end{theorem}

\noindent \textbf{Proof:} Without loss of generality, we can assume that $G$ is not complete. Like in the previous proof, we first take a Cayley graph $H$ such that $G$ is its induced subgraph; let $A \subseteq V(H)$ be the vertex set of a fixed induced copy of $G$ in $H$. Set $B = V(H) \setminus A, b = |B|, n = |V(H)|$. Let $m$ be a positive integer such that $m \geq 2b$ and $\dfrac{{m \choose 2}}{m+1} > {{n-1} \choose 2}$; as the function $\dfrac{{x \choose 2}}{x+1} = \dfrac{x(x-1)}{2(x+1)}$ is unbounded as $x$ increases, such an $m$ exists.

Now, let $K_{m,m}$ be a complete bipartite graph with parts $L = \{\ell_1,\dots, \ell_m\},$\linebreak $ R = \{r_0,r_1,\dots, r_{m-1}\}$. Add new edges $r_0v$ for every $v \in V(H)$; further, with $B = \{v_1, \dots, v_b\}$, add new edges $v_i \ell_{2i-1}, v_i \ell_{2i}$ for each $i \in \{1,\dots,b\}$. Denote the obtained graph by $F$, and its subgraph induced by $V(H) \cup \{r_0\}$ by $J$. By transitivity of $H$, all vertices of $V(H)$ in $J$ belong to the same orbit of its automorphism group, hence, they have the same betweenness $c$.

We prove first that, for every $a \in A$, ${\rm b}_F(a) = c$. Observe that, for any pair $x,y \in V(H)$, the sets of all shortest $x-y$-paths in $J$ and in $F$ are the same. Now we show that any non-trivial (that is, of length at least 2)  shortest $x-y$-path $P$ such that $y \not \in V(H)$ cannot pass through $a$.
We consider several possibilities:
\begin{itemize}
\item If $x \in V(H),y \in L$, then $d(x,y)\leq 2$, but -- as $a$ has no neighbor in $L$ -- no shortest $x-y$-path passes through $a$.
\item If $x \in V(H),y \in R \setminus \{r_0\}$, then any $x-y$-path passing through $a$ has length at least 4 (note that $a$ has no neighbors in $L$ while $y$ has neighbors only in $L$), but there is a shorter path $x r_0 \ell_1 y$. 
\item If $x,y \in L \cup R$, then the distance of $x,y$ in $F$ is at most 2 while every $x-y$-path passing through $a$ has length at least 3.
\item If $y = r_0$, then $x$ is either adjacent to $y$, or their distance is 2, but every $x-y$-path of length 2 then avoids $a$.
\end{itemize}
This implies that only pairs $x,y \in V(H)$ can give a nonzero contribution to ${\rm b}_F(a)$, and the sets of shortest $x-y$-paths in $F$ and $J$ are the same. Consequently, ${\rm b}_F(a) = {\rm b}_J(a) = c \leq {{n-1} \choose 2}$.

\medskip
Next, in $F$, consider a vertex $v \not \in A$. 
\begin{itemize}
\item If $v = v_i \in B$, then the contribution of the pair $\ell_{2i-1},\ell_{2i}$ to the value of ${\rm b}_F(v)$ is $\frac{1}{m+1}$. For a pair $x,y$ with $x,y \in V(H)$, the contribution to ${\rm b}_F(v)$ is the same as to the value of ${\rm b}_J(v)$; therefore, ${\rm b}_F(v) \geq c + \frac{1}{m+1}>c$. 
\item Further, let $v \in L$. We estimate the contribution to ${\rm b}_F(v)$ only for the pairs $x,y$ where $x,y \in R$. In $F$, there are $m$ shortest $x-y$-paths, hence, the pair $x,y$ contributes $\frac{1}{m}$ to the value of ${\rm b}_F(v)$. Since there are ${m \choose 2}$ such pairs, we obtain that ${\rm b}_F(v) \geq {m \choose 2}\cdot \frac{1}{m} > {m \choose 2}\cdot \frac{1}{m+1} > {{n-1}\choose 2} \geq  c$.
\item Finally, consider $v \in R$. We estimate the contribution to ${\rm b}_F(v)$ only for the pairs $x,y$ where $x,y \in L$. Note that $x,y$ have at most one common neighbor in $B$; we obtain that there are $m$ or $m+1$ shortest $x-y$-paths in $F$. Thus, the pair $x,y$ always contributes at least $\frac{1}{m+1}$ to ${\rm b}_F(v)$. Since there are $m \choose 2$ such pairs in total, we obtain that ${\rm b}_F(v) \geq {m \choose 2}\cdot \frac{1}{m+1} > {{n-1}\choose 2} \geq c$.
\end{itemize}

Thus, in $F$, the vertices from $A$ have smaller betweenness than the vertices outside $A$, and they induce a subgraph isomorphic to $G$. This proves the claim.
\hfill\qed

\section{Betweenness centers of trees}

The properties of betweenness centers of graphs from particular graph families are still not explored in detail (in contrast with classic eccentricity centers, see, for example, \cite{Jordan,Proskurowski1,Proskurowski2,Chang,YehChang}). We contribute to this area with the result on trees. Note that it was originally proved in \cite{HararyOstrand} in other terms for the so-called cutting center (in trees, it coincides with the betweenness center); we provide a different proof.

\begin{theorem} For any tree, its betweenness center is contained in a path. 
\end{theorem}

\noindent \textbf{Proof:}  The statement clearly holds for $K_1$ and $K_2$. Assume that there exists a tree $T$ containing three vertices $x,y,z$ of maximum betweenness which do not lie on a common path in $T$. Then there exists a unique vertex $a$ which lies on the three paths between $x,y$, $y,z$ and $x,z$. The graph $T-a$ consists of three trees $T_x,T_y,T_z$ that contain $x$, $y$ and $z$, and, possibly, other trees which do not contain these vertices. Set $k_1 = |V(T_x)|,k_2 = |V(T_y)|,k_3=|V(T)\setminus (V(T_x)\cup V(T_y))|-1$; then $k_1\geq 2, k_2 \geq 2, k_3 \geq 2$ (if, say, $k_1 = 1$, then $x$ is a pendant vertex and ${\rm b}(x) = 0$, a contradiction).   
Further, ${\rm b}(a) \geq k_1\cdot k_2 + k_2\cdot k_3 + k_1\cdot k_3$ and, also, ${\rm b}(x)\leq (k_1 -1)\cdot(k_2 + k_3+1) + \binom{k_1-1}{2} < k_1 k_2 + k_1 k_3 + k_1 - 1 + \frac{(k_1 -1)(k_1 - 2)}{2}$, ${\rm b}(y)\leq (k_2 -1)\cdot(k_1 + k_3+1) + \binom{k_2-1}{2} < k_1 k_2 + k_2 k_3 + k_2 - 1 + \frac{(k_2 -1)(k_2 - 2)}{2}$ and ${\rm b}(z)\leq (k_3 -1)\cdot(k_1 + k_2 + 1) + \binom{k_3-1}{2}<k_2 k_3 + k_1 k_3 + k_3 - 1 + \frac{(k_3 -1)(k_3 - 2)}{2}$. As $x,y,z$ have maximum betweenness, ${\rm b}(a)\leq {\rm b}(x) ={\rm b}(y)={\rm b}(z)$. It follows that
\[
\begin{aligned}
k_1 k_2 + k_2 k_3 + k_1 k_3 &\leq k_1 k_2 + k_1 k_3 + k_1 - 1 + \frac{(k_1 -1)(k_1 - 2)}{2} \\
k_1 k_2 + k_2 k_3 + k_1 k_3 &\leq k_1 k_2 + k_2 k_3 + k_2 - 1 + \frac{(k_2 -1)(k_2 - 2)}{2}
\\
k_1 k_2 + k_2 k_3 + k_1 k_3 &\leq k_2 k_3 + k_1 k_3 + k_3 - 1 + \frac{(k_3 -1)(k_3 - 2)}{2}
\end{aligned}
\]
or, equivalently,  
\[
\begin{aligned}
2k_2 k_3 &\leq k_1(k_1-1)
\\
2k_1 k_3 &\leq k_2(k_2-1)
\\
2k_1 k_2 &\leq k_3(k_3-1).
\end{aligned}
\]
Since $k_i k_j < 2k_ik_j$ and $k_i(k_i-1) < k_i^2$, we obtain
\[
\begin{aligned}
k_2 k_3 &< k_1^2
\\
k_1 k_3 &< k_2^2
\\
k_1 k_2 &< k_3^2.
\end{aligned}
\]
Multiplying the corresponding sides of these inequalities yields $k_1^2 k_2^2 k_3^2 < k_1^2 k_2^2 k_3^2 $, a contradiction. 
\hfill $\square$ 

\medskip
Note that an analogous result does not hold for unicyclic graphs: taking three stars $K_{1,5}$ and joining the center of each with a vertex of $K_3$, we obtain a graph with a unique cycle, whose betweenness center consists precisely of the centers of the three attached stars.

\medskip
Up to order 20, there is no tree with a betweenness center of order greater than 4. Table 1 shows the numbers of $n$-vertex trees ($n = 1,\dots, 20$) having a $k$-vertex betweenness center, for $k = 1,\dots, 4$. Although the results of \cite{HararyOstrand} yield the existence of trees of arbitrarily large betweenness centers (in addition, with arbitrarily prescribed structure within a common path), their construction is highly non-trivial, involving existential approaches from calculus; moreover, it seems that, for fixed tree order, such trees are extremely rare.

\begin{table}[h!]
\[
\begin{array}{|r|r|r|r|r|r|}
\hline
n & k=1 & k=2 & k=3 & k=4 & {\rm total\ count} \\
\hline
\hline
1 & 1 & 0 & 0 & 0 &  1 \\
\hline
2 & 0 & 1 & 0 & 0  & 1 \\
\hline
3 & 1 & 0 & 0 & 0  & 1 \\
\hline
4 & 1 & 1 & 0 & 0 &  2 \\
\hline
5 & 3 & 0 & 0 & 0 &  3 \\
\hline
6 & 4 & 2 & 0 & 0 &  6 \\
\hline
7 & 9 & 1 & 1 & 0 &  11 \\
\hline
8 & 18 & 5 & 0 & 0  & 23 \\
\hline
9 & 45 & 2 & 0 & 0  & 47 \\
\hline
10 & 93 & 10 & 2 & 1  & 106 \\
\hline
11 & 225 & 10 & 0 & 0  & 235 \\
\hline
12 & 512 & 39 & 0 & 0  & 551 \\
\hline
13 & 1244 & 56 & 1 & 0  & 1301 \\
\hline
14 & 2958 & 191 & 9 & 1  & 3159\\
\hline
15 & 7477 & 253 & 11 & 0  & 7741 \\
\hline
16 & 18266 & 997 & 54 & 3  & 19320 \\
\hline
17 & 47524 & 1100 & 5 & 0  & 48629 \\
\hline
18 & 119502 & 4263 & 101 & 1 & 123867 \\
\hline
19 & 309467 & 8336 & 151 & 1 & 317955 \\
\hline
20 & 800104 & 21289 & 1386 & 286 & 823065 \\
\hline
\end{array}
\]
\caption{The census of trees with betweenness centers of prescribed size}
\end{table}

\section{Concluding remarks}

The fact that the trees considered above have relatively small -- sometimes very small -- betweenness centers might suggest that this holds in general. In light of this, we formulate

\begin{problem}
Let $
f(n)=\max\{|C_{\rm b}(T)|:\ T\text{ is a tree of order }n\}
$.
Determine $f(n)$, or at least its asymptotic order of growth. Is $f(n) = o(n)$?
\end{problem}

\begin{conjecture}
Let $T_n$ be chosen uniformly at random from all unlabeled trees of order $n$. Then
$
\Pr(|C_{\rm b}(T_n)|=1)\longrightarrow 1
\quad\text{as }n\to\infty.
$
\end{conjecture}

The examples 2 and 3 show that the betweenness center of a graph need not be connected. This motivates the following 

\begin{problem}
Find sufficient conditions for graphs to have connected betweenness centers.
\end{problem}

\bigskip\noindent
{\bf Acknowledgement. } This research was supported by the Slovak Research and Development Agency under the Contract No. APVV-23-0191.






\bibliographystyle{elsarticle-num-names-alphsort}



\end{document}